\documentclass[11pt]{article} 

\usepackage{amssymb,amsthm,amsmath}

\usepackage[top=1in, bottom=1.5in, left=1in, right=1in]{geometry}

\usepackage[T1]{fontenc}
\usepackage[utf8]{inputenc}

\newcommand{\R}{\mathbb{R}}

\newcommand{\CC}{\mathcal C}

\newcommand{\norm}[1]{\left\lVert#1\right\rVert}

\newcommand{\ex}[1]{\mathsf{E}}

\newcommand{\Ex}[1]{\mathsf{E}\Big[#1\Big]}

\newcommand{\dif}{\mathsf{d}}
\newcommand{\difr}{\mathsf{d}^{\mathsf R}}
\newcommand{\difI}{\mathsf{d}^{\mathsf{I}}}

\newtheorem{theorem}{Theorem}
\newtheorem{corollary}{Corollary}
\newtheorem{proposition}{Proposition}
\newtheorem{lemma}{Lemma}
\newtheorem{remark}{Remark}
\newtheorem{definition}{Definition}

\title{The relation between mixed and rough SDEs  and its application to numerical methods} 

\author{Andreas Neuenkirch\\
{\small\textsl{Institut f\"ur Mathematik, Universit\"at Mannheim}}\\
{\small\textsl{A5, 6, D-68131 Mannheim, Germany}}\\
{\small\texttt{email: neuenkirch@kiwi.math.uni-mannheim.de}}
\\
Taras Shalaiko\\
{\small\textsl{Institut f\"ur Mathematik, Universit\"at Mannheim}}\\
{\small\textsl{A5, 6, D-68131 Mannheim, Germany}}\\
{\small\texttt{email:tshalaik@uni-mannheim.de}}
\\
 }

\begin{document}
\maketitle
\vspace*{0.5cm}

\noindent {\bf Abstract}\quad {\em 
We study the relationship between mixed stochastic differential equations and  the corresponding rough path equations driven by standard Brownian motion and  fractional Brownian motion with Hurst parameter $H>1/2$. We establish 
a correction formula, which relates both types of equations, analogously to the It\=o-Stratonovich correction formula. This correction formula allows to transfer properties, which are established for one type of equation to the other, and
 we will illustrate this by considering numerical methods for mixed and rough SDEs.}

\section{Introduction}
In this manuscript we will consider two types of stochastic  differential equations (SDEs), the so-called mixed SDEs (see e.g. \cite{kubilius,missh11}) 
and rough SDEs (see e.g. \cite{Lyons-bk, Friz}) driven by standard Brownian motion and fractional Brownian motion with Hurst parameter $H>1/2$.
The mixed SDE reads as
\begin{align}
 X^{M}_t=x_0+\int_0^t a^M(X^M_s)\dif s+\int_0^t b^M(X^M_s) \difI W_s+\int_0^t c^M(X^M_s)\dif B^H_s, \qquad t \geq 0 ,
\label{eqM}
\end{align}
while the corresponding rough path equation is given by
\begin{align}
 X^{R}_t=x_0+\int_0^t a^R(X^R_s)\dif s+\int_0^t b^R(X^R_s) \difr W_s+\int_0^t c^R(X^R_s)\difr B^H_s, \qquad t \geq 0,
 \label{eqRough}
\end{align}
where $W=(W_t)_{t\geq 0}$ is an $m$-dimensional Wiener process, $B^H=(B^H_t)_{t\geq  0 }$ is an $\ell$-dimensional fractional Brownian motion with Hurst index $H>1/2$, $x_0\in \R^d$, and the coefficients $a^M,a^R\colon \R^d\to \R^d$, $b^M,b^R\colon \R^d\to \R^{d\times m}$, $c^M,c^R\colon \R^d\to \R^{d\times \ell}$ satisfy  suitable smoothness assumptions. The precise definitions are given below.\\
The difference between both equations is the definition of the stochastic integrals with respect to the Brownian motion. When dealing with mixed equations, $\int_0^t b^M(X^M_s) \difI W_s$ is understood as
It\=o integral, while for rough equations $\int_0^t b^R(X^R_s) \difr W_s$ corresponds to a Stratonovich integral.
Both equations have been well studied  so far;  in \cite{mbfbm-limit,Shev-Delay} the unique solvability of the mixed SDE was shown, while the well definedness and unique solvability of the rough path SDE have
been obtained  e.g. in \cite{CQ}.

In this manuscript we establish a correction formula (see Sections \ref{rptomixed} and  \ref{mixedtorp})  between equations \eqref{eqM} and \eqref{eqRough}, which extends the It\=o-Stratonovich correction formula that goes back to two articles of W. Wong and M. Zakai (\cite{WZ1,WZ2}).
This allows to transfer results valid for rough SDEs to mixed SDEs and vice versa. We  will illustrate this by establishing the smoothness of the solution map  and by recovering a limit theorem for the mixed SDE \eqref{eqM}
in Section \ref{rptomixed_appl}; furthermore, we point out how this limit theorem can be used to construct and analyse numerical methods for mixed SDEs,
and we show that the ``natural'' Euler scheme for the rough SDE converges to the rough solution.

\medskip

\section{Preliminaries}
\label{Prelim}
\subsection{Notation and Definitions}

In what follows we will work on a filtered probability space $\left(\Omega, \mathcal F, (\mathcal F_t)_{t\geq 0} , \mathbb P\right)$, which is rich enough to contain all the objects defined below.
Let $W= (W^{(1)}_t,\ldots, W^{(m)}_t )_{t\geq 0}$ be a standard $m$-dimensional Wiener process, 
$B^H= (B^{H,(1)},\ldots,B^{H,(\ell)} )_{t\geq 0}$  be an $\ell$-dimensional fractional Brownian motion (fBm) with Hurst index $H\in (1/2,1)$, that is a collection of centered, independent Gaussian processes, independent of $W$ as well, with covariance function
$$
R_H(t,s)=\frac 1 2 \left(t^{2H}+s^{2H}-|t-s|^{2H} \right), \qquad s,t \geq 0.
$$
The Kolmogorov theorem entails that fBm has a modification with $\gamma$-H\"older sample paths for any $\gamma<H$ and we will identify $B^H$ with this modification in the following.


We will use the following standard notation: $|\cdot|$ stands for an absolute value of a real number, the Euclidean norm of a finite dimensional vector or of a matrix. Moreover, for a function $f\colon [a,b]\to \R$ we define the following (semi-)norms:
\begin{align*}
 \norm{f}_{\infty,[a,b]}& =\sup_{x\in [a,b]}|f(x)|,  \qquad 
 \norm{f}_{\gamma,[a,b]}=\sup_{\substack{x,y\in[a,b]\\x\neq y}}\frac{|f(x)-f(y)|}{|x-y|^\gamma},\\
 & \norm{f}_{\gamma,\infty,[a,b]}=\norm{f}_{\gamma,[a,b]}+\norm{f}_{[a,b],\infty}.
\end{align*}
If  there is no ambiguity we will omit an interval index $[a,b]$. For a vector valued function $f=(f_1,\ldots,f_d)\colon [a,b]\to \R^d$ a corresponding (semi-) norm is defined as the sum of the (semi-) norms of the coordinates $f_i$. 
Next, for a function $f\colon [a,b]^2\to \R$, that vanishes on a diagonal, i.e. $f(t,t)=0$ for $t\in [a,b]$, we set 
$$
\norm{f}_{\gamma,[a,b]^2}=\sup_{\substack{t,s\in[a,b]\\t\neq s}}\frac{|f(t,s)|}{|t-s|^\gamma}.
$$
The set of such functions with a finite norm $\norm{f}_{\gamma,[a,b]^2}$ is denoted by $\CC^\gamma_2([a,b]^2)$, that is
 $\CC^\gamma_2([a,b]^2)=\{f\colon [a,b]^2\to \R \mid f(t,t)=0,\,t\in [a,b],\, \norm{f}_{\gamma,[a,b]^2}<\infty \}$. 
 
 Moreover, we will use the notation $C_{b}^{k,\delta}(\R^{d_{1}}; \R^{d_2})$ for functions $f: \R^{m_{1}} \rightarrow \R^{m_2}$, which are bounded, $k$-times differentiable with bounded derivatives and whose $k$-th derivative is H\"older continuous of order $\delta >0$.
 Finally, if a map $c: \R^k \rightarrow \R^{k,m}$ is fixed, we introduce the differential operators $\mathcal D^{(i)}_c=\sum_{l=1}^k c_{l,i}^M(\cdot)\partial_{x_l}$.

\subsection{Integrating with respect to standard Brownian motion and  fBm with $H>1/2$}\label{integratingfbm}

For basic facts for It\=o or Stratonovich integration with respect to standard Brownian motion we refer e.g. to \cite{KS,KP}.

The integral with respect to  a fractional Brownian motion with Hurst parameter $H>1/2$ is understood in the pathwise Young sense. Namely, for $f\in \CC^{\nu}([a,b]; \R)$ and $g\in \CC^{\mu}([a,b];\R)$ with $\nu+\mu>1$ the integral $\int_a^b f(x)\dif g(x)$ can be defined as the limit of its Riemann sums and satisfies the so-called Young inequality
$$        \left|  \int_{a}^{b} (f(s)-f(a) ) \dif  g(s) \right| \leq C_{\nu, \mu }  \| f\|_{\nu, [0,T]} \|g\|_{\mu,[0,T]} |b-a|^{\lambda+ \mu} $$  
for all $a,b \in [0,T]$, where  $C_{\nu,\mu}>0$ is a constant independent of $f$ and $g$. 
 Thus, the integral $\int_a^b f(s)\dif B^{H,(i)}_s$  for a function $f\in \CC^{\beta}([a,b];\R)$ is well defined provided that $\beta>1-H$. More details  can be found  e.g. in \cite{young,Friz}.

\subsection{Mixed SDEs}
The mixed equation \eqref{eqM} reads as
\begin{align} \label{mixed_eq_2}
 X^{M}_t=x_0+\int_0^t a^M(X^M_s)\dif s& +\sum\limits_{j=1}^m \int_0^t b^{M,(j)}(X^{M}_s)\difI W^{(j)}_s \\ &+
 \sum\limits_{j=1}^{\ell}\int_0^t c^{M,(j)}(X^M_s)\dif B^{H,(j)}_s, \qquad t \in [0,T], \nonumber 
\end{align}
where  $(\cdot)^{(j)}$ denotes the $j$-th column of a matrix. As mentioned previously the integrals with respect to the Brownian motions are understood as It\=o integrals, while the integrals 
with respect to the fractional Brownian motions are understood as Young integrals. This equation has been analysed in a series of articles (\cite{kubilius,missh11,  
mbfbm-limit,Shev-Delay}). The most general result on the existence of a  unique solution
can be found in \cite{Shev-Delay}:

\begin{theorem}\label{mixed:main_result}
Assume that
\begin{itemize}
  \item[(i)]  $a^M,  b^{M,(i)},  c^{M,(j)}\in C^{1}(\R^d; \R^d)$, $i=1, \ldots, m$, $j=1, \ldots, \ell$,
   \item [(ii)] $a^M,  b^{M,(i)}, c^{M,(j)}$, $i=1, \ldots, m$, $j=1, \ldots, \ell$, satisfy a linear growth condition, i.e. there exists $C>0$ such that
  $$|a^M(x)| +   \sum_{i=1}^m|b^{M,(i)}(x)|   + \sum_{j=1}^{\ell} |c^{M,(j)}(x)| \leq C(1+ |x|), \qquad x \in \R^d, $$
  \item[(iii)]  $ \sup_{j=1, \ldots, \ell} \sup_{x \in \R^d} |(\operatorname{D} c^{M,(j)})(x)| < \infty$.
\end{itemize}
Then equation \eqref{eqM} has a unique solution, i.e.~there exists a unique continuous and $(\mathcal{F}_t)_{t \in [0,T]}$ adapted process $X=(X_t)_{t \in [0,T]}$, which satisfies  equation \eqref{eqM}
for almost all $\omega \in \Omega$.
\end{theorem}

The above solution is in fact obtained as the limit (in probability) of the solutions of  It\=o SDEs with random coefficients, namely of
$$
 X^{M,n}_t=x_0+\int_0^t \big (a^M(X^{M,n}_s)+c^M(X^{M,n}_s) \dot{B}^{H,n}_s \big ) \dif s+\int_0^t b^M(X^M_s) \difI W_s, \qquad t \in [0,T],
$$
where $B^{H,n}_t=n\int_{(t- 1/n)\vee 0}^t B^H_s \dif s$, $t \in [0,T]$, $n=1,2, \ldots$, is a smoothed fBm, see \cite{mbfbm-limit}.

\subsection{Rough paths}
Here we briefly recall some notions of the rough path theory, following the algebraic integration approach
given in \cite{Guba} and the recent monograph \cite{Hai}.
For a detailed exposition the reader is sent to \cite{Lyons-bk,Friz,Guba,Hai}.
\begin{definition} Let $\gamma >1/3$.
  A pair $(x,\mathbf{x})=(x_s,\mathbf{x}_{s,t})_{0\leq t,s\leq T}\in \CC^{\gamma}([0,T]; \R^m)\times \CC^{2\gamma}_2([0,T]^2; \R^{m\times m})$ is called a $\gamma$-rough path if $\mathbf{x}_{t,t}=0$ for
  all $t\in [0,T]$ and for all $0\leq s<u<t\leq T$ we have
  $$
  \mathbf{x}_{s,t}-\mathbf{x}_{s,u}-\mathbf{x}_{u,t}=(x_s-x_u)\otimes(x_t-x_u).
  $$
\end{definition}

The function $(\mathbf{x}_{s,t})_{s,t,\in [0,T]}$ is called L\'evy area.

\begin{remark} \label{rem_levy} \begin{itemize}
                \item[(i)]            
 A L\'evy area for the fractional Brownian motion $B^H$  with Hurst parameter $H>1/2$ is defined as the collection of Young integrals $$\mathbf{B}=( \mathbf{B}_{s,t})_{0\leq s <t\leq T}=\left\{\int_s^t\int_s^u \dif B^{H,(i)}_v\dif B^{H,(j)}_u; \, 1\leq i,j\leq \ell\right \}_{0\leq s<t\leq T}. $$
\item[(ii)] The L\'evy area for standard Brownian  can be constructed using  Stratonovich integration, i.e. $$\mathbf{W}=(\mathbf{W}_{s,t})_{0\leq s<t\leq T}= \left \{\int_{s}^t \int_s^u \circ\, \dif W^{(i)}_v \circ\dif W^{(j)}_u; \, 1\leq i,j\leq m \right\}_{0\leq s<t\leq T},$$
 where $\int \circ \dif W$ denotes the  Stratonovich integral. 
 \item[(iii)] In the same way a L\'evy area for $(t,W_t,B_t)_{t \in [0,T]}$ can be constructed, i.e. using $\mathbf{W}$ for the iterated integrals of $W^{(i)}$ with respect to $W^{(j)}$, $i,j=1, \ldots,m$, and Young integrals for all other iterated integrals.
 \end{itemize}
\end{remark}

Together with the notion of the L\'evy area the following concept is at the core of the algebraic integration approach of M. Gubinelli.

\begin{definition}
 We say that a path $y\in \CC^\nu([0,T]; \R^k)$ with $ \nu \in (1/3,\gamma]$ is a weakly controlled path based on $x\in \CC^{\gamma}([0,T]; \R^m)$ if the following decomposition holds
 \begin{align} \label{weak:dcp}
 y_t-y_s=z_s(x_t-x_s)+r_{s,t}, \qquad 0 \leq s \leq t \leq T,
 \end{align}
 with $z\in \CC^{\nu}([0,T]; \R^{k\times m})$ and $r\in \CC^{2\nu}_2([0,T]^2;\R^k)$.
\end{definition}
When $(y,z)$ is a weakly controlled path, then the rough integral of $y$ along $x$ can be defined as 
\begin{align} \label{crucial-1}
\int_0^t y^{(i)}\dif x^{(j)}_s=\lim_{|\mathcal P|\to 0}\sum_{t_k\in \mathcal P} \left( y^{(i)}_{t_k}(x^{(j)}_{t_{k+1}}-x^{(j)}_{t_k})+\sum_{\ell=1}^m z_{t_k}(i,\ell)\mathbf{x}_{t_k,t_{k+1}}(\ell,j) \right),
\end{align}
for $i=1, \ldots, k$, $j=1,\ldots, m$, see e.g.~Corollary 2 in \cite{Guba}.
Here the limit is taken over all partitions $\mathcal{P}=\{0=t_{-1}=t_0 <t_1 < \ldots <t_n=t_{n+1}=t\}$ such that $|\mathcal{P}|= \sup_{t_k \in \mathcal{P}} |t_{k}-t_{k-1}| \rightarrow 0$. Weakly controlled paths are stable under smooth transformations:

\begin{proposition}\label{cp:weak-phi}
Let $ (y,z)$ be a weakly controlled path based on $x$ with decomposition (\ref{weak:dcp}), and let
$\varphi\in   C^{2}_b(\R^k;\R^n)$. Then $\varphi(y)$ is  a weakly controlled path based on $x$ with decomposition
$$ \varphi(y_t)-\varphi(y_s)= \hat{z}_s  (x_t-x_s) +\hat{r}_{s,t},$$
with
$$
\hat{z}_s= (\operatorname{D} \varphi)(y_s) z_s, \qquad s \in [0,T].
$$  
\end{proposition}

Using appropriate estimates
for the integrals the solution to the rough paths equation
$$ d y_t= \sigma(y_t) \dif x_t, \quad t \in [0,T], \qquad y_0 =a \in \R^d,
 $$ with $ \sigma: \R^d\to \R^{d\times m}$ is obtained via a fixed point argument.

\label{roughpaths}
\begin{theorem}
 Suppose that $\kappa \in (1/3 ,\gamma)$, $x\colon [0,T]\to \R^m$ is a $\gamma$-rough path and  let $\sigma\in C_b^{2,\delta}(\R^d;\R^{d\times m})$ such that
 $(2 +\delta)\gamma >1$. Then the equation
 $$
 y_t=a+\int_0^t \sigma(y_s)\dif x_s, \qquad t \in [0,T],
 $$
 possesses a unique solution in the space of the functions $z\in \CC^\kappa([0,T]; \R^d)$ with $z_0=a$. Moreover, $(y,\sigma(y))$ is a weakly controlled path based on $x$.
\end{theorem}

As a consequence of this Theorem  and Proposition \ref{cp:weak-phi} we have
\begin{align} \label{crucial} &
\int_0^t \sigma^{(i)}(y_s)\dif x^{(i)}_s \\ \nonumber & \qquad =\lim_{|\mathcal P|\to 0}\sum_{t_k\in \mathcal P} \left( \sigma^{(i)}(y_{t_k})(x^{(i)}_{t_{k+1}}-x^{(i)}_{t_k})+\sum_{\ell=1}^m \mathcal D^{(\ell)}_{\sigma} \sigma^{(i)}(y_{t_k})\mathbf{x}_{t_k,t_{k+1}}(\ell,i)\right).
\end{align}

The solution map for a rough equation is locally Lipschitz continuous with respect to the initial value and the driving signal. More precisely, we have:
\begin{theorem}\label{stability} Let $\kappa \in (1/3,\gamma)$, $\sigma\in C_b^{2,\delta}(\R^d;\R^{d\times m})$ such that
 $(2 +\delta)\gamma >1$, $a,\tilde{a}\in \R^d$,  and let $x$, $\tilde{x}$ be $\gamma$-rough paths with  corresponding L\'evy areas $\mathbf{x}$, $\tilde{\mathbf{x}}$.
Finally, let $(y_t)_{t\in[0,T]}$, $(\tilde{y}_t)_{t\in[0,T]}$ be the solutions of  the RDEs
 \begin{align*}
  y_t=a+\int_0^t \sigma(y_s)\dif x_s, \quad
  \tilde{y}_t=\tilde{a}+\int_0^t\sigma(\tilde{y}_s)\dif \tilde{x}_s, \qquad t \in [0,T].
 \end{align*}
Then there exist an increasing function $C_T\colon [0, \infty) \to [0, \infty)$ such that
\begin{align*}
 \norm{y-\tilde y}_{\gamma,\infty,[0,T]}\leq C_T\big(\norm{x}_{\gamma,\infty,[0,T]}+\norm{\tilde x}_{\gamma,\infty,[0,T]}+\norm{\mathbf{x}}_{\gamma,[0,T]}+\| \tilde{\mathbf{x}}\|_{\gamma,[0,T]} \big)\\
 \times (|a-\tilde a|+\norm{x-\tilde x}_{\gamma,\infty,[0,T]}+\|\mathbf{x}-\tilde{\mathbf{x}}\|_{2\gamma,[0,T]}).
\end{align*}
\label{convTheorem}
\end{theorem}

Returning to our original rough SDE, i.e. to
\begin{align}\label{Req2}
 X^{R}_t=x_0+\int_0^t a^R(X^R_s)\dif s& +\sum\limits_{i=1}^m \int_0^t b^{R,(i)}(X^{R}_s)\difr W^{(i)}_s  \\ & +\sum\limits_{j=1}^{\ell}\int_0^t c^{R,(j)}(X^R_s)\difr B^{H,(j)}_s, \qquad t \in [0,T], \nonumber
\end{align}
 the previous results yield a unique solution, if $a^R,  b^{R,(i)}, c^{R,(j)} \in C_b^{2,\delta}(\R^d;\R^d)$, $i=1, \ldots, m$, $j=1, \ldots, \ell$,
for $\delta >0$ arbitrarily small, where the L\'evy area for $(t,W_t,B_t)_{t \in [0,T]}$ is constructed as in Remark \ref{rem_levy}.

\section{From the rough paths equation to the mixed equation}
\label{rptomixed}
Here we show that the solution of the rough SDE \eqref{Req2} is the solution of the mixed equation \eqref{mixed_eq_2} with
 \begin{align} a^M(x)=a^R(x)+\frac 1 2 \sum_{i=1}^d \mathcal D^{(i)}_{b^{R}} b^{R,(i)}(x), \, \,  b^M(x)=b^R(x), \,\, c^M(x)=c^R(x), \,\,x \in \R^d. \label{coeff_rel_1} \end{align}

\begin{theorem} Let  $\delta >0$ and $a^R,  b^{R,(i)}, c^{R,(j)} \in C_b^{2,\delta}(\R^d;\R^d)$, $i=1, \ldots, m$, $j=1, \ldots, \ell$. Then the solution $X^R$ of the rough equation  \eqref{Req2}
and the  solution $X^M$ of the mixed equation  \eqref{mixed_eq_2} with coefficients given by \eqref{coeff_rel_1} coincide $P$-almost surely, i.e. we have
$$ P\big(X_t^M=X_t^R, \, t \in [0,T]\big)=1.$$ \label{rptomixed_thm}
\end{theorem} 
\begin{proof}
For the $m+\ell+1$ dimensional rough path $g=(\operatorname{id},W,B^H)$ denote its L\'evy area by $\mathbf{G}=(\mathbf{G}_{s,t})_{0\leq s<t\leq T}$. Now fix $t \in [0,T]$. Using relation \eqref{crucial} for  $\sigma=(a^R,b^R,c^R)\colon \R^d\to \R^d\times \R^{d\times m}\times \R^{d\times \ell}$ we can write 
\begin{align*}
\int_0^t \sigma^{(i)}(X^R_s)\difr g^{(i)}_s=\lim_{|\mathcal P|\to 0}\sum\limits_{t_k\in  \mathcal P} \left( \sigma^{(i)} (X^R_{t_k}) g^{(i)}_{t_k,t_{k+1}}+\sum\limits_{j=1}^{1+m+\ell} \mathcal D^{(j)}_{\sigma} \sigma^{(i)}(X^R_{t_k})\mathbf G_{t_k,t_{k+1}}(j,i)\right )
\end{align*}
for all $i=1,\ldots, 1+m+\ell$. 

Since $g^{(1)}=\operatorname{id}$ and the integrand is continuous, we have
$$ \lim_{|\mathcal P|\to 0}\sum\limits_{t_k\in\mathcal P}\sigma^{(1)} (X^R_{t_k}) g^{(1)}_{t_k,t_{k+1}}  \stackrel{P-a.s.}{=} \int_0^t a^R(X^R_s)\dif s. $$
For $i=1, \ldots, m$, we have
$$ \lim_{|\mathcal P|\to 0}\sum\limits_{t_k\in\mathcal P}\sigma^{(i+1)} (X^R_{t_k}) g^{(i+1)}_{t_k,t_{k+1}}  \stackrel{L^2(\Omega)}{=} \int_0^t b^{R,(i)}(X^R_s) \difI W^{(i)}_s, $$
by definition of the It\=o integral since $b^{R,(i)}(X^R_s)$, $s \in [0,T]$, is bounded and adapted. Moreover, the sample paths of $X^{R}$ are $\gamma$-H\"older continuous for all $\gamma<1/2$ and $B^H$
are H\"older continuous of all orders $\lambda<H$, thus we have 
$$ \lim_{|\mathcal P|\to 0}\sum\limits_{t_k\in\mathcal P}\sigma^{(i+m+1)} (X^R_{t_k}) g^{(i+m+1)}_{t_k,t_{k+1}}  \stackrel{P-a.s.}{=} \int_0^t c^{R,(i)}(X^R_s)\dif B^{H,(i)}_s $$
for $i=1, \ldots, \ell$ by definition of the Young integral.

Now consider the summands involving the L\'evy area terms. If $(i,j)\not \in (2,\ldots,m+1)^2$, then
the Young inequality gives
$$ \sup_{\substack{t,s\in[0,T]\\t\neq s}}
\frac{ \left|\int_s^t (g^{(i)}_u-g^{(i)}_s)\dif g^{(j)}_u \right|}{|t-s|^{1+\varepsilon}} < \infty \qquad P-a.s.$$
and hence it follows
 $$\lim_{|\mathcal P|\to 0}\sum\limits_{t_k\in\mathcal P}\mathcal D^{(j)}_{\sigma} \sigma^{(i)}(X^R_{t_k})\mathbf G_{t_k,t_{k+1}}(j,i) \stackrel{P-a.s.}{=}0.$$
 Next suppose $(i,j)\in (2,\ldots,m+1)^2$ and $i\neq j$, then 
$\mathbf G_{t_k,t_{k+1}}(i,j)=\int_s^t (W^{(i)}_u-W^{(i)}_s)\difI W^{(j)}_u$, since the Stratonovich and the It\=o integral coincide due to the independence of $W^{(i)}$, $W^{(j)}$. Exploiting the independence
of $G_{t_k,t_{k+1}}$ from $\mathcal{F}_{t_k}$
we obtain
\begin{align*}
& \Ex{\Big|\sum\limits_{t_k\in\mathcal P} \mathcal D^{(j)}_{\sigma}\sigma^{(i)}(X^R_{t_{k}})\mathbf G_{t_k,t_{k+1}}(j,i)\Big|^2}\\ & \qquad=\sum\limits_{t_k\in \mathcal P} \Ex{|\mathcal D^{(j)}_{\sigma}\sigma^{(i)}(X^R_{t_k})|^2\Big (\int_{t_k}^{t_{k+1}} (W^{(j)}_u-W^{(j)}_{t_k})\dif W^{(i)}_u\Big)^2}\\ & \qquad  =
\sum\limits_{t_k\in \mathcal P} \Ex{|\mathcal D^{(j)}_{\sigma}\sigma^{(i)}(X^R_{t_k})|^2}\Ex{\Big(\int_{t_k}^{t_{k+1}} (W^{(j)}_u-W^{(j)}_{t_k})\dif W^{(i)}_u\Big)^2} \\ & \qquad  =
\frac 1 2 \sum\limits_{t_k\in \mathcal P} \Ex{|\mathcal D^{(j)}_{\sigma}\sigma^{(i)}(X^R_{t_k})|^2}|\mathcal P|^2.
\end{align*}
The last term clearly vanishes for $|\mathcal P|\to 0$. Finally, if $i=j$ we
have $$ \int_{t_k}^{t_{k+1}} (W^{(i)}_u-W^{(i)}_{t_k}) \circ\dif W^{(i)}_u = \frac{1}{2} (W^{(i)}_{t_{k+1}}-W^{(i)}_{t_k})^2$$
and we obtain
\begin{align*} \lim_{|\mathcal P|\to 0 } \sum\limits_{t_k\in \mathcal P} \mathcal D^{(i)}_{\sigma} \sigma^{(i)}(X^R_{t_k})(W^{(i)}_{t_{k+1}}-W^{(i)}_{t_k})^2 \stackrel{P-a.s.}{=} \int_0^t  \mathcal D^{(i)}_{\sigma} \sigma^{(i)}(X^R_{s}) \dif s .
\end{align*}
 The latter follows from 
 $$  \sup_{s \in [0,T]} \left| \sum_{k=0}^{n-1}{\bf 1}_{[0,s]}(t_k)(W^{(i)}_{t_{k+1}}-W^{(i)}_{t_k})^2-s\right|\stackrel{P-a.s.}{\longrightarrow} 0$$ as  $|\mathcal P|\to 0$ and a density argument.

 By passing to a subsequence, we deduce that 
 \begin{align*}
\sum_{i=1}^{m+\ell+1}\int_0^t \sigma^{(i)}(X^R_s)\difr g^{(i)}_s 
 & \stackrel{P-a.s.}{=} \int_0^t\Big ( a^R(X^R_s)+\frac 1 2 \sum_{i=1}^m \mathcal D^{(i)}_{b^R} b^{R,(i)}(X^R_s)\Big) \dif s\\ & \qquad +\sum_{i=1}^m \int_0^t b^{R,(i)}(X^{(R)}_s)\difI W^{(i)}_s+\sum_{i=1}^{\ell} \int_0^t c^{R,(i)}(X^R_s)\dif B^{H,(i)}_s
 \end{align*}
 for all $t \in [0,T]$.  Since both sides are continuous in $t$ for almost all $\omega \in \Omega$, the exceptional set can be chosen independently of $t \in [0,T]$.
 Hence the assertion follows. 
\end{proof}

\section{From the mixed equation to  the rough paths equation}
\label{mixedtorp}
Throughout this section we assume that $a^M,b^M,c^M\in C^2_b$. Under this assumption the solution $X^M=(X^M_t)_{t\in[0,T]}$ to \eqref{mixed_eq_2}  exists, is unique and satisfies
$ \mathsf{E} \| X\|_{\theta}^p < \infty$ for all $p \geq 1$ and $\theta <1/2$, see
\cite{Shev-Delay}. Now, we will show that $X^M$ is the solution to \eqref{Req2} with the  coefficients 
\begin{align} \label{coeff_rel_2}
a^R(x)=a^M(x)-\frac 1 2 \sum_{i=1}^m\mathcal D^{(i)}_{b^M}b^{M,(i)}(x),\,\, b^R(x)=b^M(x),\,\, c^R(x)=c^M(x), \,\, x \in \R^d.
\end{align}
In order to do to this, we have to show that 
$(X^M(\omega), \sigma(X^M(\omega))$ is a weakly controlled path based on $\{g(t)(\omega), \mathbf{G}_{t,s}(\omega)\}_{0\leq s<t\leq T}$ for almost all $\omega \in \Omega$, where
$\sigma=(a^M,b^M,c^M)$. However, this is a consequence of the following two Lemmata.

\begin{lemma}
Let $h\in C^2(\R^d; \R)$. Then for all $\gamma\in (0,1/2)$ there exist  almost surely finite random variables $K_{T,h,\gamma,g}$ such that
$$
\left |\int_s^t (h(X^M_u)-h(X^M_s)) \dif g(u) \right|\leq K_{T, h, \gamma,g} \cdot  |t-s|^{2\gamma}, \qquad s,t \in [0,T], $$
where $ g \in \{ \operatorname{id}, B^{H,(i)}\}$ with $i \in \{1, \ldots, \ell\}$.
\end{lemma}
\begin{proof}
Since $X^M\in C^\gamma([0,T])$ for all $\gamma <1/2$, the assertion is a direct consequence of
the Young inequality, which gives
 $$ \left|\int_s^t (h(X^M_u)-h(X^M_s)) \dif g(u) \right| \leq C_{\gamma,\gamma'} |t-s|^{\gamma + \gamma'} \| f(X^M) \|_{\gamma,[0,T]} \|g \|_{\gamma',[0,T]}$$
for $1-\gamma < \gamma' <H$.
\end{proof}

\begin{lemma}
\label{estimInt}
 Suppose $(Z(t),\mathcal F_t)_{t\in [0,T]}$ is a stochastic process with $\theta$-H\"older trajectories for all $\theta\in (0,1/2)$, such that 
 $ \mathsf{E} \|Z \|_{\theta}^p <\infty $ for all $p\geq 1$. Then,  for all $\eta \in (0, \theta)$, there exists an almost surely finite random variable $K_{T,\eta}$ such that
$$\left|\int_s^t (Z(v)-Z(s))\difI W_v \right|\leq K_{T,\eta}|t-s|^{1/2+\eta}, \qquad s,t \in [0,T].$$
\end{lemma}

\begin{proof}
 Let $\theta \in (0,1/2)$. Applying the Rodemich-Garcia-Rumsey inequality, we obtain that 
 $$ \sup_{\substack{t,s\in[0,T]\\t\neq s}}  \frac{ \left|\int_s^t (Z(v)-Z(s))\difI W_v \right|}{|t-s|^{1/2+\eta}} \leq C_{\theta,\eta ,p} \left(\int_s^t\int_s^t \frac{|\int_x^y (Z(v)-Z(s)) \difI W_v|^{2p}}{|x-y|^{ (1 + 2\eta) p +2}} \dif x\dif y\right)^{1/2p}.$$
 Now put  $$ K_{T ,\eta }=\left(\int_s^t\int_s^t \frac{|\int_x^y (Z(v)-Z(s)) \difI W_v|^{2p}}{|x-y|^{ (1 + 2\eta) p +2}} \dif x\dif y\right)^{1/2p}.$$
 The Burkholder-Davis-Gundy inequality gives
 $$   \mathsf{E} \left|\int_x^y (Z(v)-Z(s)) \difI W_v \right|^{2p} \leq C_{\theta,p} |y-x|^{(1+2\theta)p}.     $$
Choosing $p$ such that $2(\theta- \eta)p   > 1$ we obtain
 \begin{align*}
  \mathsf{E} K_{T,\eta}^{2p} &= \int_s^t\int_s^t  \frac{ \mathsf{E} |\int_x^y (Z(v)-Z(s)) \difI W_v|^{2p}}{|x-y|^{ (1 + 2\eta) p +2}} \dif x\dif y \leq  C_{\theta,p}
  \int_s^t\int_s^t |x-y|^{ -2(\theta- \eta)p -2} \dif x \dif y  < \infty .
  \end{align*}
\end{proof}

Now we can exploit the representation \eqref{crucial-1} and Proposition  \ref{cp:weak-phi} to work backwards through the proof of  Theorem  \ref{rptomixed_thm}, which gives:  

\begin{theorem}\label{transfer_mixed_rough} Let $a^M,b^{M,(i)},c^{M,(j)}\in C^{2}_b([0,T]; \R^d)$, $i=1, \ldots, m$, $j=1, \ldots, \ell$, and
suppose that $X^M=(X^M_t)_{t\in[0,T]}$ is the solution to \eqref{mixed_eq_2}, then $X^M$ is a solution to the rough equation \eqref{Req2}  with  coefficients
given by  \eqref{coeff_rel_2}.
\end{theorem}

Note that the  smoothness of the  drift coefficient of the arising rough path SDE is $C^1_b$. Within  the algebraic integration framework it is not known (up to the best of our knowledge) whether such an equation has a unique solution.

\section{Application to numerical methods} \label{rptomixed_appl}

\subsection{Limit theorem for mixed equations}
At the core of the theory of mixed equations is a limit theorem in \cite{mbfbm-limit}, which we will briefly recall here. Let $n \in \mathbb{N}$ and define $$ B^{H,n}_t=n\int_{(t- 1/n)\vee 0}^t B^H_s \dif s, \qquad t \geq 0. $$ Note that $B^{H,n}$ is an $(\mathcal{F}_t)_{t \geq 0}$-adapted Gaussian process such that
its
trajectories  are a.s.~differentiable and  $$ \dot{B}^{H,n}_t=n(B^H_t-B^H_{(t-1/n)\vee 0}), \qquad t \geq 0.$$ 
Suppose that $X^{M,n}=(X^{M,n}_t)_{t\in [0,T]}$ is a solution to
\begin{align} \label{mixed_approx}
 X^{M,n}_t=x_0+\int_0^t \big( a(X^{M,n}_s)+c(X^{M,n}_s)  \dot{B}^{H,n}_s \big) \dif s+\int_0^t b(X^M_s) \difI W_s, \qquad t \in [0,T].
\end{align}
Then we have
\begin{align}
X^{M,n}\to X^M,\,n\to \infty\,\,\, \text{uniformly in probability}.
\label{result}
\end{align}
 Setting $g^n=(\operatorname{id},W,B^{H,n})$, we can apply the stability of rough paths equations and the relation between mixed and rough SDEs to recover and strengthen this result.

First note that the integral $\int_0^t c(X^{R,n}) \difr B^{H,n}_s$ coincides with the ordinary Young integral $\int_0^t c(X^{R,n}_s)\dot{B}_s^{H,n} \dif s$.
Proceeding analogously to the proof of Theorem \ref{rptomixed_thm} we have:

\begin{proposition} Let  $\delta >0$ and $a,  b^{(i)}, c^{(j)} \in C_b^{2,\delta}(\R^d;\R^d)$, $i=1, \ldots, m$, $j=1, \ldots, \ell$. Then the solution $X^{R,n}$ of the rough equation  
 \begin{align}  \label{eqSmoothed}
 X^{R,n}_t=x_0 & +\int_0^t  \widetilde{a}(X^{R,n})  \dif s   +\int_0^t b(X^{R,n}) \difr W_s  +\int_0^t c(X^{R,n}) \difr B^{H,n}_s, \quad t \in [0,T],
\end{align} with
$$\widetilde{a}(x)=a(x) - \frac 1 2 \sum_{i=1}^m \mathcal D^{(i)}_b b^{(i)}(x), \qquad x \in \R^d, $$
and the solution of  equation \eqref{mixed_approx} coincide $P$-almost surely.
 \end{proposition}

Our aim is now to prove:

\begin{proposition}\label{aux_prop} Let  $\tilde{a},b^{(i)},c^{(j)}\in C^{2,\delta}_b(\R^d; \R^d)$, $i=1, \ldots, m$, $j=1, \ldots, \ell$, and $\gamma \in (1/3, 1/2)$. Then we have
$$ \| X^{R,n}- X^R \|_{\gamma,    \infty, [0,T]} \stackrel{P-a.s.} \longrightarrow 0 \qquad \textrm{as} \quad n \rightarrow \infty. $$
\label{conv}
\end{proposition}
This result directly implies:
\begin{corollary}\label{app_mixed_cor} Let  $a,b^{(i)},c^{(j)}\in C^{2,\delta}_b(\R^d; \R^d)$, $i=1, \ldots, m$, $j=1, \ldots, \ell$, $\gamma \in (1/3, 1/2)$ and
$  \sum_{i=1}^m \mathcal D^{(i)}_b b^{(i)}  \in C^{2,\delta}_b([0,T]; \R^d)$. Then  we have
$$ \| X^{M,n}- X^M \|_{\gamma, \infty, [0,T]} \stackrel{P-a.s.} \longrightarrow 0 \qquad \textrm{as} \quad n \rightarrow \infty. $$
\label{conv}
\end{corollary}

Recalling that  $g^n=(\operatorname{id},W,B^{H,n})$ Proposition \ref{aux_prop}  follows from Theorem \ref{stability} and the following two Lemmata. (Note that the following estimates are not covered by Section 15.5 in \cite{Friz}, since
$B^{H,n}$ is not a mollifier approximation.)

\begin{lemma}
For all $0 <\gamma< \gamma' < H$  there exists an almost surely finite random variable $K_{T,\gamma, \gamma'}$ such that
$$
\| B^{H,n}- B^H \|_{\gamma, [0,T]} \leq K_{T,\gamma, \gamma'} \cdot n^{-(\gamma'-\gamma)}.
$$
\end{lemma}
\begin{proof}
Clearly, it is sufficient to consider the one dimensional case.  For $ t \leq 0$ define $B^H_t=0$. Fix $t,s\in [0,T]$ and $\gamma'\in (\gamma,H)$. First, consider the case $|t-s|\geq \frac 1 n$. Here, we have
 \begin{align*}
  |B^H_t-B^H_s-B^{H,n}_t+B^{H,n}_s| & \leq n \left|\int^t_{t-1/n}(B^H_u-B^H_t) \dif u\right|+n\left|\int_{s-1/n}^s (B^H_u-B^H_s) \dif u\right| \\ & \leq
  2\norm{B^H}_{\gamma',[0,T]} \frac{1}{n^{\gamma'}}\leq 2 \norm{B^H}_{\gamma',[0,T]}|t-s|^{\gamma} \frac{1}{n^{\gamma-\gamma'}}.
\end{align*}
Next, when $|t-s|\leq \frac 1 n $ one has
 \begin{align*}
  |B^H-B^H_s-B^{H,n}_t+B^{H,n}_s| & \leq  |B^H_t-B^H_s|+|B^{H,n}_t-B^{H,n}_s|\\
&  \leq \norm{B^H}_{\gamma',[0,T]}|t-s|^{\gamma'} +n \left|\int_{-\frac 1 n}^0 (B^{H}_{t+u}-B^{H}_{s+u}) \dif u \right| \\ & \leq 
 2\norm{B^H}_{\gamma',[0,T]}|t-s|^{\gamma'}\leq 2\norm{B^H}_{\gamma',[0,T]}|t-s|^{\gamma} \frac{1}{n^{\gamma'-\gamma}}.
\end{align*}
\end{proof}

\begin{lemma}
 Let  $1/2 <\gamma < \gamma' <H$. Then, there exists an almost surely finite random variable $K_{T,\gamma,\gamma'}$ such that 
 $$
  \| \mathbf G^n - \mathbf{G} \|_{\CC^\gamma_2([0,T]^2)} \leq K_{T,\gamma,\gamma'} \cdot n^{-(\gamma'-\gamma)} .
 $$
\end{lemma}
\begin{proof}
 First, we prove the convergence the elements of the L\'evy areas which correspond to the smoothed fBm to the ones of fBm. Fix $t,s\in [0,T]$, $i,j\in \{1,\ldots,\ell\}$. Clearly, we have
 \begin{align*} 
 \Delta^{(1)}(s,t) & = \left  |\int_s^t (B^{H,(i)}_u-B^{H,(i)}_s)\dif B^{H,(j)}_u-\int_s^t( B^{H,n,(i)}_u-B^{H,n,(i)}_s)\dif B^{H,n,(j)}_u \right |
 \\ &  \leq
 \left |\int_s^t \big( ( B^{H,(i)}_u-B^{H,(i)}_s)-(B^{H,n,(i)}_u-B^{H,n,(i)}_u) \big) \dif B^{H,(j)}_u \right|  
 \\ & \qquad   +\left|\int_s^t ( B^{H,n,(i)}_u-B^{H,n,(i)}_s) \dif (B^{H,n,(j)}_u-B^{H,(j)}_u)\right|.
 \end{align*}
Set $Z_u^{n,(i)}=B^{H,(i)}_u-B^{H,n,(i)}_u$. Applying the Young inequality with $1/2< \mu  <H$  yields
\begin{align*}
 \left |\int_s^t (Z^n_u-Z^n_s)\dif B^{H,(j)}_u \right | & \leq C_{\mu} \norm{B^{H,(j)}}_{\mu}\norm{Z^{n,(i)}}_{\mu}|t-s|^{2\mu}, \\
 \left |\int_s^t (B^{H,n,(i)}_u-B^{H,n,(i)}_s)\dif (B^{H,n,(j)}_u-B^{H,(j)}_u)\right| & \leq C_\mu \norm{B^{H,n,(i)}}_\mu \norm{Z^{n,(j)}}_{\mu}|t-s|^{2\mu}, 
 \end{align*} i.e. we obtain
  \begin{align} 
 \frac{\Delta^{(1)}(s,t)}{|t-s|^{2\mu}} \leq  C_{\mu} \left( \norm{Z^{n,(i)}}_{\mu} \norm{B^{H,(j)}}_{\mu} +  \norm{Z^{n,(j)}}_{\mu} \norm{B^{H,(i)}}_{\mu} + \norm{Z^{n,(j)}}_{\mu} \norm{Z^{n,(i)}}_{\mu}
  \right). 
\end{align}

Now, we proceed with the parts that correspond to the  iterated integrals which involve the Wiener process and the smoothed fBm. Fix 
$i\in\{1,\ldots, m\}$ and $j\in\{1,\ldots, \ell\}$, $s,t\in [0,T]$. Again, applying  the Young inequality gives
\begin{align*}
\Delta^{(2)}(s,t)& =\left |\int_s^t (W^{(i)}_u-W^{(i)}_s) \dif (B^{H,(j)}_u -B^{H,n,(j)}_u) \right|  \\ & \leq
 C_{\lambda,\mu} \norm{Z^{n,(j)}}_\mu \norm{W^{(i)}}_\lambda |t-s|^{\lambda + \mu}
\end{align*}
with $0<\lambda<1/2, 0< \mu< H$ such that $\lambda + \mu >1$.
It is only left to deal with
\begin{align*} \Delta^{(3)}(s,t)&= 
\left |\int_s^t(B^{H,(j)}_u-B^{H,(j)}_s)\dif W^{(i)}_u-\int_s^t (B^{H,n,(j)}_u-B^{H,n,(j)}_s) \dif W^{(i)}_u \right |.
\end{align*}
But using the integration by parts formula for Young integrals we obtain
\begin{align*}
 \int_s^t (Z_u^{n,(j)}-Z_s^{n,(j)})\dif W^{(i)}_u=(Z^{n,(j)}_t-Z^{n,(j)}_s)(W^{(i)}_t-W^{(i)}_s)-\int_s^t( W^{(i)}_u-W^{(i)}_s )\dif Z^{n,(j)}_u.
\end{align*}
Using the previous step yields
$$
\Delta^{(3)}(s,t) \leq C_{\lambda,\mu} \norm{Z^{n,(j)}}_\mu \norm{W^{(i)}}_\lambda |t-s|^{\lambda + \mu}.
$$

The elements of the L\'evy area involving $t$ and the ``smoothed'' fBms are easily treated. Here we have
\begin{align*}
\Delta^{(4)}(s,t):= \left| \int_s^t (B_s^{H,(i)} - B_s^{H,n,(i)}) \dif s\right| \leq \norm{Z^{n,(i)}}_{\mu} |t-s|^{1+\mu}, \\
\Delta^{(5)}(s,t):= \left| \int_s^t s \dif  (B_s^{H,(i)} - B_s^{H,n,(i)})  \right| \leq \norm{Z^{n,(i)}}_{\mu} |t-s|^{1+\mu}.
\end{align*}
Setting now $\gamma=(\lambda +  \mu)/2$, the assertion follows from the  previous lemma. 
\end{proof}

\subsection{Constructing numerical methods for mixed equations}

The almost sure convergence in the ${\gamma}$-H\"older norm of
\begin{align} 
 X^{M,n}_t=x_0+\int_0^t \big( a(X^{M,n}_s)+c(X^{M,n}_s)  \dot{B}^{H,n}_s \big) \dif s+\int_0^t b(X^M_s) \difI W_s, \qquad t \in [0,T], \label{mixed_num_2}
\end{align}
with
$$  \dot{B}^{H,n}_t=   n(B^H_t-B^H_{(t-1/n)\vee 0}), \qquad t \in [0,T], $$ to $X^M$
can be exploited to construct and to analyse numerical methods for mixed equations, proceeding similar to \cite{DNT,Riedel}. In the latter references  approximations of rough SDEs have been obtained by discretising their Wong-Zakai approximations. For example, 
applying an Euler discretisation with stepsize $\Delta=1/n$ to \eqref{mixed_num_2} yields the approximation
$$ x_{k+1} = x_{k} + a(x_k) \Delta  + b(x_k)(W_{(k+1)\Delta} - W_{k\Delta}) + c(x_k)(B^H_{k\Delta} - B^H_{(k-1)\Delta}), \qquad k=0,1, \ldots $$
where $B^H_{-\Delta}=0$.
Equation \eqref{mixed_num_2} is an It\^{o} SDE with random coefficients,  the only technical difficulty being the unboundedness of $\dot{B}^{H,n}$. Using a localization procedure as e.g. in \cite{num_math}
and  standard estimates involving the It\=o isometry and Gronwall's lemma one can show that
$$  \sup_{k=0, \ldots ,\lceil T/n \rceil} |  X^{M,n}_{k \Delta} - x_{k}| \stackrel{P-a.s.}{\longrightarrow} 0, \qquad  n \rightarrow \infty.$$ Corollary \ref{app_mixed_cor} then implies the convergence of this skewed Euler scheme, i.e.
$$ \sup_{k=0, \ldots, \lceil T/n \rceil} |  X^{M}_{k \Delta} - x_{k}| \stackrel{P-a.s.}{\longrightarrow} 0, \qquad n \rightarrow \infty.$$

\subsection{The natural Euler scheme for rough SDEs}
Using the correction formula, one can establish the convergence of the ``natural'' Euler scheme
\begin{align}  \label{euler_rp}  x_{k+1} = x_{k} & + \Big( a(x_k)  + \frac 1 2 \sum_{i=1}^m \mathcal D^{(i)}_b b^{(i)}(x_k) \Big) \Delta \\ 
& + b(x_k)(W_{(k+1)\Delta} - W_{k\Delta}) + c(x_k)(B^H_{(k+1)\Delta} - B^H_{k\Delta}), \qquad k=0,1, \ldots \nonumber  \end{align}
for the rough SDE
\begin{align} \label{rough_num_2} 
 X^{R}_t=x_0+\int_0^t a(X^{R}_s) \dif s + \int_0^t b(X^R_s) \difr W_s + \int_0^t c(X^R_s) \difr B_s^H, \qquad t \in [0,T], 
\end{align}
at least for $m=\ell=1$.
The notion ``natural'' is based on the following observations: for $b=0$ equation \eqref{rough_num_2} is an SDE driven by fractional Brownian motion with Hurst parameter $H>1/2$, for which \eqref{euler_rp}
with $b=0$ is a convergent approximation, see e.g. \cite{Davie,Friz}, while for $c=0$  equation \eqref{rough_num_2} is a Stratonovich SDE, for which \eqref{euler_rp} with $c=0$ is again a convergent scheme, see e.g. \cite{KP}.

Using the results of \cite{ShevMishura-Euler} and Theorem \ref{transfer_mixed_rough} we have:

\begin{proposition} Let $a,b,c\in C^{2,\delta}_b(\R; \R)$. Moreover let
$  b'b  \in C^{2,\delta}_b(\R; \R)$ and $\inf_{x \in \mathbb{R}} c(x)>0$. Then there exists $C>0$ such that
$$ \sup_{k=0, \ldots, \lceil T/ \Delta \rceil} \left( \mathsf{E} | X_{k \Delta}^R - x_k|^2 \right)^{1/2} \leq C \cdot \big( \Delta^{1/2} + \Delta^{2H-1}  \big).$$
\end{proposition}




\end{document}